\documentclass[a4paper,10pt]{amsart}
\usepackage[T1]{fontenc}
\usepackage{lmodern}
\usepackage[all]{xy}
\usepackage{amssymb,enumitem,nicefrac,mathtools}
\usepackage[british]{babel}
\newcommand*{\MRrefBentmann}[2]{\linebreak[0] \href{http://www.ams.org/mathscinet-getitem?mr=#1}{MR \textbf{#1}}}

\usepackage[pdftitle={Projective dimension in filtrated K-theory},
  pdfauthor={Rasmus Bentmann},
  pdfsubject={Mathematics; MSC 18G20, 19K35, 46L80}
]{hyperref}
\usepackage[lite]{amsrefs}
\usepackage{microtype}

\usepackage{tikz}
\usetikzlibrary{matrix,arrows}
\def\mathclap{\mathpalette\mathclapinternal}
\def\mathclapinternal#1#2{\clap{$\mathsurround=0pt#1{#2}$}}
\let\oldbigoplusBentmann=\bigoplus
\def\alternativebigoplusBentmann{\oldbigoplusBentmann\limits}

\DeclareMathOperator{\HomBentmann}{Hom}
\DeclareMathOperator{\PrimBentmann}{Prim}
\DeclareMathOperator{\ExtBentmann}{Ext}
\DeclareMathOperator{\TorBentmann}{Tor}
\DeclareMathOperator{\kerBentmann}{ker}
\DeclareMathOperator{\cokerBentmann}{coker}
\DeclareMathOperator{\imBentmann}{im}
\DeclareMathOperator{\diagBentmann}{diag}

\newcommand*{\KKBentmann}{\textup{KK}}
\newcommand*{\KBentmann}{\textup{K}}
\newcommand*{\SyzBentmann}{\textup{Syz}}
\newcommand*{\HomologyBentmann}{\textup{H}}
\newcommand*{\FKBentmann}{\textup{FK}}
\newcommand*{\KKBentmanncat}{\mathfrak{KK}}
\newcommand*{\idealBentmann}{{\mathfrak{I}}}
\newcommand*{\AbBentmann}{\mathfrak{Ab}}      
\newcommand*{\ModBentmannc}[1]{\mathfrak{Mod}\bigl(#1\bigr)_\textup{c}}
\newcommand*{\opBentmann}{\mathrm{op}}
\newcommand*{\ConeBentmann}{\mathrm{Cone}}
\newcommand*{\ZBentmann}{\mathbb{Z}}
\newcommand*{\NBentmann}{\mathbb{N}}
\newcommand*{\QBentmann}{\mathbb{Q}}
\newcommand*{\nbBentmann}{\nobreakdash}  
\newcommand*{\CstarBentmann}{\texorpdfstring{$C^*$\nbBentmann-}{C*-}}
\newcommand*{\blankBentmann}{\text{\textvisiblespace}}
\newcommand*{\defeqBentmann}{\mathrel{\vcentcolon=}}
\newcommand*{\ininBentmann}{\mathrel{{\in}{\in}}}
\newcommand*{\LCBentmann}{\mathbb{LC}}
\newcommand*{\NTBentmann}{\mathcal{NT}}
\newcommand*{\NTBentmannC}{{\mathcal{NT}^*}}
\newcommand*{\NTBentmannnil}{\mathcal{NT}_\textup{nil}}
\newcommand*{\NTBentmannss}{\mathcal{NT}_\textup{ss}}
\newcommand*{\MssBentmann}{M_\textup{ss}}
\newcommand*{\sss}{\textup{ss}}
\newcommand*{\bigbBentmann}[1]{\bigl(#1\bigr)}
\newcommand*{\idBentmann}{\textup{id}}
\newcommand*{\RepBentmann}{\mathcal{R}}
\newcommand*{\CstBentmann}{C^*}
\newcommand*{\CuntzBentmann}{\mathcal O}
\newcommand*{\epiBentmann}{\twoheadrightarrow}

\newcommand*{\SphereBentmann}{\mathbb S}
\newcommand*{\BootBentmann}{\mathcal{B}}

\newdir{ >Bentmann}{{}*!/-5pt/@{>}}

\newtheorem{propositionBentmann}{Proposition}

\newtheorem{lemmaBentmann}{Lemma}

\newtheorem{remarkBentmann}{Remark}
\newtheorem{exampleBentmann}{Example}

\title[Projective dimension in filtrated K-theory]{Projective dimension in filtrated K-theory}

\author{Rasmus Bentmann}

\address{Department of Mathematical Sciences\\University of
Copenhagen\\Universitetsparken 5\\2100 Copenhagen \O \\Denmark}
\email{bentmann@math.ku.dk}

\thanks{The author was supported by the Danish National Research Foundation through the Centre for Symmetry and Deformation and by the Marie Curie Research Training Network EU-NCG}

\begin{document}

\begin{abstract}
Under mild assumptions, we characterise modules with projective resolutions of length $n\in\NBentmann$ in the target category of filtrated $\KBentmann$-theory over a finite topological space in terms of two conditions involving certain $\TorBentmann$-groups. We show that the filtrated $\KBentmann$-theory of any separable \CstarBentmann algebra over any topological space with at most four points has projective dimension~2 or less. We observe that this implies a universal coefficient theorem for rational equivariant $\KKBentmann$-theory over these spaces. As a contrasting example, we find a separable \CstarBentmann algebra in the bootstrap class over a certain five-point space, the filtrated $\KBentmann$-theory of which has projective dimension~3. Finally, as an application of our investigations, we exhibit Cuntz-Krieger algebras which have projective dimension~2 in filtrated $\KBentmann$-theory over their respective primitive spectrum.
\end{abstract}

\maketitle

\section{Introduction}
A far-reaching classification theorem in \cite{KirchbergBentmann} motivates the computation of Eberhard Kirchberg's ideal-related Kasparov groups \index{$KKBentmann(X)$-theory} $\KKBentmann(X;A,B)$ for separable \CstarBentmann al\-ge\-bras $A$ and $B$ over a non-Hausdorff topological space $X$ by means of $\KBentmann$-theoretic invariants. We are interested in the specific case of finite spaces here. In \cites{MN:HomologicalIBentmann, MN:FiltratedBentmann}, Ralf Meyer and Ryszard Nest laid out a theoretic framework that allows for a generalisation of Jonathan Rosenberg's and Claude Schochet's universal coefficient theorem \cite{RSBentmann} to the equivariant setting. Starting from a set of generators of the equivariant bootstrap class, they define a homology theory with a certain universality property, which computes $\KKBentmann(X)$-theory via a spectral sequence. In order for this \index{universal coefficient sequence} \emph{universal coefficient} spectral sequence to degenerate to a short exact sequence, it remains to be checked \emph{by hand} that objects in the range of the homology theory admit projective resolutions of length~1 in the Abelian target category.

Generalising earlier results from \cites{Bonkat:ThesisBentmann, Restorff:ThesisBentmann,MN:FiltratedBentmann} the verification of the above-mentioned condition for \index{filtrated $\KBentmann$-theory} \emph{filtrated $\KBentmann$-theory} was achieved in \cite{BKBentmann} for the case that the underlying space is a disjoint union of so-called accordion spaces. A finite connected $T_0$-space $X$ is an accordion space if and only if the directed graph corresponding to its specialisation pre-order is a Dynkin quiver of type~A. Moreover, it was shown in \cites{MN:FiltratedBentmann,BKBentmann} that, if $X$ is a finite $T_0$-space which is not a disjoint union of accordion spaces, then the projective dimension of filtrated $\KBentmann$-theory over~$X$ is \emph{not} bounded by~$1$ and objects in the equivariant bootstrap class are \emph{not} classified by filtrated $\KBentmann$-theory. The assumption of the separation axiom $T_0$ is not a loss of generality in this context (see \cite{MN:BootstrapBentmann}*{\S2.5}).

There are two natural approaches to tackle the problem arising for non-accordion spaces: one can either try to refine the invariant---this has been done with some success in \cite{MN:FiltratedBentmann} and \cite{BentmannBentmann}; or one can hold onto the invariant and try to establish projective resolutions of length~1 on suitable subcategories or localisations of the category $\KKBentmanncat(X)$, in which $X$-equivariant $\KKBentmann$-theory is organised. The latter is the course we pursue in this note. We state our results in the next section.

\subsection*{Acknowledgement} Parts of this paper are based on the author's Diplom thesis~\cite{BentmannBentmann} which was supervised by Ralf Meyer at the University of G\"ottingen. I would like to thank my PhD-supervisors,  S\o ren Eilers and Ryszard Nest, for helpful advice, Takeshi Katsura for pointing out a mistake in an earlier version of the paper and the anonymous referee for the suggested improvements.

\section{Statement of Results}
The definition of filtrated $\KBentmann$-theory and related notation are recalled in~\S\ref{sec:preliminariesBentmann}.  

\begin{propositionBentmann}
	\label{pro:projdim1Bentmann}
Let $X$ be a finite topological space. Assume that the ideal $\NTBentmannnil\subset\NTBentmannC(X)$ is nilpotent and that the decomposition $\NTBentmannC(X)=\NTBentmannnil\rtimes\NTBentmannss$ holds. Fix $n\in\NBentmann$. For an $\NTBentmannC(X)$-module $M$, the following assertions are equivalent:
\begin{enumerate}[label=\textup{(\roman*)}]
\item
$M$ has a projective resolution of length~$n$.
\item
The Abelian group $\TorBentmann_n^{\NTBentmannC(X)}(\NTBentmannss,M)$ is free and the Abelian group\newline $\TorBentmann_{n+1}^{\NTBentmannC(X)}(\NTBentmannss,M)$ vanishes.
\end{enumerate}
\end{propositionBentmann}
The basic idea of this paper is to compute the $\TorBentmann$-groups above by writing down projective resolutions for the fixed right-module $\NTBentmannss$.

Let $Z_m$ be the $(m+1)$-point space on the set $\{1,2,\ldots,m+1\}$ such that $Y\subseteq Z_m$ is open if and only if $Y\ni m+1$ or $Y=\emptyset$. A \(\CstBentmann\)\nbBentmann-algebra over~\(Z_m\) is a \(\CstBentmann\)\nbBentmann-algebra~\(A\) with a distinguished ideal such that the corresponding quotient decomposes as a direct sum of~$m$ orthogonal ideals. Let~$S$ be the set $\{1,2,3,4\}$ equipped with the topology $\{\emptyset,4,24,34,234,1234\}$, where we write $24\defeqBentmann\{2,4\}$ etc. A \CstarBentmann{}algebra over~$S$ is a \CstarBentmann{}algebra together with two distinguished ideals which need not satisfy any further conditions; see \cite{MN:BootstrapBentmann}*{Lemma 2.35}.

\begin{propositionBentmann}
	\label{pro:projdim2Bentmann}
Let $X$ be a topological space with at most~\textup{4} points. Let $M=\FKBentmann(A)$ for some \CstarBentmann{}algebra~$A$ over~$X$. Then $M$ has a projective resolution of length~\textup{2} and $\TorBentmann_2^{\NTBentmannC}(\NTBentmannss,M)=0$.

Moreover, we can find explicit formulas for $\TorBentmann_1^{\NTBentmannC}(\NTBentmannss,M)$; for instance, $\TorBentmann_1^{\NTBentmannC(Z_3)}(\NTBentmannss,M)$ is isomorphic to the homology of the complex
\begin{equation}
	\label{eq:homology_computing_Tor1Bentmann}
\bigoplus_{j=1}^3 M(j4)
\xrightarrow{\left(\begin{smallmatrix}i&-i&0\\-i&0&i\\0&i&-i\end{smallmatrix}\right)} \bigoplus_{k=1}^3 M(1234\setminus k)
\xrightarrow{\left(\begin{smallmatrix}i&i&i\end{smallmatrix}\right)} M(1234)\;.
\end{equation}
A similar formula holds for the space~$S$; see \eqref{eq:tor_sequence_for_SBentmann}.
\end{propositionBentmann}

The situation simplifies if we consider \emph{rational} \index{$KKBentmann(X)$-theory} $\KKBentmann(X)$-theory, whose morphism groups are given by $\KKBentmann(X;A,B)\otimes\QBentmann$; see \cite{MR2812030Bentmann}. This is a $\QBentmann$-linear triangulated category which can be constructed as a localisation of $\KKBentmanncat(X)$; the corresponding localisation of filtrated $\KBentmann$-theory is given by $A\mapsto\FKBentmann(A)\otimes\QBentmann$ and takes values in the category of modules over the $\QBentmann$-linear category $\NTBentmannC(X)\otimes\QBentmann$.
\begin{propositionBentmann}
  \label{pro:rationalBentmann}
Let $X$ be a topological space with at most~\textup{4} points. Let $A$ and $B$ be \CstarBentmann{}algebras over $X$. If $A$ belongs to the equivariant bootstrap class $\BootBentmann(X)$, then there is a natural short exact \index{universal coefficient sequence} universal coefficient sequence
\begin{eqnarray*}
\ExtBentmann_{\NTBentmannC(X)\otimes\QBentmann}^1\bigl(\FKBentmann_{\ast+1}(A)\otimes\QBentmann,\FKBentmann_{\ast}(B)\otimes\QBentmann\bigr) \rightarrowtail \KKBentmann_{\ast}(X;A,B)\otimes\QBentmann\\ \twoheadrightarrow \HomBentmann_{\NTBentmannC(X)\otimes\QBentmann}\bigl(\FKBentmann_{\ast}(A)\otimes\QBentmann,\FKBentmann_{\ast}(B)\otimes\QBentmann\bigr)\;.
\end{eqnarray*}
\end{propositionBentmann}
In \cite{MR2812030Bentmann}, a long exact sequence is constructed which in our setting, by the above proposition, reduces the computation of $\KKBentmann_*(X;A,B)$, up to extension problems, to the computation of a certain torsion theory $\KKBentmann_*(X;A,B;\QBentmann/\ZBentmann)$.

The next proposition says that the upper bound of~2 for the projective dimension in Proposition~\ref{pro:projdim2Bentmann} does not hold for all finite spaces.
\begin{propositionBentmann}
	\label{pro:projdim3Bentmann}
There is an $\NTBentmannC(Z_4)$-module $M$ of projective dimension~\textup{2} with free entries and $\TorBentmann_2^{\NTBentmannC}(\NTBentmannss,M)\neq 0$. The module $M\otimes_\ZBentmann \ZBentmann/k$ has projective dimension~\textup{3} for every $k\in\NBentmann_{\geq 2}$. Both~$M$ and $M\otimes_\ZBentmann \ZBentmann/k$ can be realised as the \index{filtrated $\KBentmann$-theory} filtrated $\KBentmann$-theory of an object in the equivariant bootstrap class~$\BootBentmann(X)$.
\end{propositionBentmann}

As an application of Proposition~\ref{pro:projdim2Bentmann} we investigate in~\S\ref{sec:Cuntz-KriegerBentmann} the obstruction term $\TorBentmann_1^{\NTBentmannC}\bigbBentmann{\NTBentmannss,\FKBentmann(A)}$ for certain Cuntz-Krieger algebras with four-point primitive ideal spaces. We find:
\begin{propositionBentmann}
	\label{pro:CKBentmann}
There is a \index{Cuntz-Krieger algebra} Cuntz-Krieger algebra with primitive ideal space homeomorphic to~$Z_3$ which fulfills Cuntz's condition \textup{(II)} and has projective dimension~\textup{2} in filtrated $\KBentmann$-theory over~$Z_3$. The analogous statement for the space~$S$ holds as well.
\end{propositionBentmann}
The relevance of this observation lies in the following: \emph{if} Cuntz-Krieger algebras \emph{had} projective dimension at most~1 in filtrated $\KBentmann$-theory over their primitive ideal space, this would lead to a  strengthened version of Gunnar Restorff's classification result~\cite{RestorffBentmann} with a proof avoiding reference to results from symbolic dynamics.

\section{Preliminaries}
\label{sec:preliminariesBentmann}
Let $X$ be a finite topological space. A subset $Y\subseteq X$ is called \emph{locally closed} if it is the difference $U\setminus V$ of two open subsets~$U$ and~$V$ of~$X$; in this case, $U$~and~$V$ can always be chosen such that $V\subseteq U$. The set of locally closed subsets of $X$ is denoted by $\LCBentmann(X)$. By $\LCBentmann(X)^*$, we denote the set of \emph{non-empty, connected} locally closed subsets of~$X$.

Recall from \cite{MN:BootstrapBentmann} that a \index{\CstarBentmann algebra over~\(X\)} \emph{\CstarBentmann algebra over~\(X\)} is pair \((A,\psi)\) consisting of a \CstarBentmann{}al\-ge\-bra~\(A\) and a continuous map \(\psi\colon \PrimBentmann(A)\to X\).
A \CstarBentmann algebra \((A,\psi)\) over~\(X\) is called \emph{tight} if the map $\psi$ is a homeomorphism.
A \CstarBentmann algebra \((A,\psi)\) over~\(X\) comes with \emph{distinguished subquotients} $A(Y)$ for every $Y\in\LCBentmann(X)$.

There is an appropriate version $\KKBentmann(X)$ of \index{$KKBentmann(X)$-theory} bivariant $\KBentmann$-theory for \CstarBentmann{}al\-ge\-bras over~$X$ (see \cites{KirchbergBentmann,MN:BootstrapBentmann}). The corresponding category, denoted by $\KKBentmanncat(X)$, is equipped with the structure of a triangulated category (see~\cite{NeemanBentmann}); moreover, there is an equivariant analogue $\BootBentmann(X)\subseteq\KKBentmanncat(X)$ of the bootstrap class~\cite{MN:BootstrapBentmann}.

Recall that a triangulated category comes with a class of distinguished candidate triangles. An \emph{anti-distinguished} triangle is a candidate triangle which can be obtained from a distinguished triangle by reversing the sign of one of its three morphisms. Both distinguished and anti-distinguished triangles induce long exact $\HomBentmann$-sequences.

As defined in \cite{MN:FiltratedBentmann}, for $Y\in\LCBentmann(X)$,
we let $\FKBentmann_Y(A)\defeqBentmann\KBentmann_*\bigl(A(Y)\bigr)$ denote the \(\ZBentmann/2\)-graded $\KBentmann$-group of the
subquotient of $A$ associated to~$Y$.
Let~\(\NTBentmann(X)\) be the \(\ZBentmann/2\)-graded pre-additive category whose
  object set is \(\LCBentmann(X)\) and whose space of morphisms from \(Y\) to \(Z\)
  is \(\NTBentmann_*(X)(Y,Z)\) -- the  \(\ZBentmann/2\)-graded Abelian group of all natural transformations
  \(\FKBentmann_Y\Rightarrow \FKBentmann_Z\).
Let~\(\NTBentmannC(X)\) be the full subcategory with object set \(\LCBentmann(X)^*\). We often abbreviate $\NTBentmannC(X)$ by $\NTBentmannC$.

Every open subset of a locally closed subset of $X$ gives rise to an extension of distinguished subquotients. The corresponding natural maps in the associated six-term exact sequence yield morphisms in the category $\NTBentmann$, which we briefly denote by $i$, $r$ and~$\delta$.

A \emph{\textup (left-\textup)module} over~\(\NTBentmann(X)\) is a grading-preserving,
additive functor from \(\NTBentmann(X)\) to the category \(\AbBentmann^{\ZBentmann/2}\) of \(\ZBentmann/2\)-graded Abelian groups. A morphism of $\NTBentmann(X)$-modules is a natural transformation of functors. Left-modules over $\NTBentmannC(X)$ are defined similarly. By \(\ModBentmannc{\NTBentmannC(X)}\) we denote the category of countable $\NTBentmannC(X)$-modules.

\index{filtrated $\KBentmann$-theory} \emph{Filtrated $\KBentmann$-theory} is the functor $\KKBentmanncat(X)\to\ModBentmannc{\NTBentmannC(X)}$ taking a \CstarBentmann{}algebra $A$ over $X$ to the collection $\bigl(\KBentmann_*(A(Y))\bigr)_{Y\in\LCBentmann(X)^*}$ with the obvious $\NTBentmannC(X)$-module structure.

Let $\NTBentmannnil\subset\NTBentmannC$ be the ideal generated by all natural transformations between different objects, and let $\NTBentmannss\subset\NTBentmannC$ be the subgroup spanned by the identity transformations $\idBentmann_Y^Y$ for objects $Y\in\LCBentmann(X)^*$. The subgroup $\NTBentmannss$ is in fact a subring of $\NTBentmannC$ isomorpic to $\ZBentmann^{\LCBentmann(X)^*}$. We say that $\NTBentmannC$ decomposes as semi-direct product $\NTBentmannC=\NTBentmannnil\rtimes\NTBentmannss$ if $\NTBentmannC$ as an Abelian group is the inner direct sum of $\NTBentmannnil$ and $\NTBentmannss$; see \cites{BKBentmann, MN:FiltratedBentmann}. We do not know if this fails for any finite space.

We define \emph{right-modules} over~\(\NTBentmannC(X)\) as \emph{contravariant}, grading-preserving,
  additive functors \(\NTBentmannC(X)\to\AbBentmann^{\ZBentmann/2}\).
If we do not specify between left and right, then we always mean left-modules.
The subring $\NTBentmannss\subset\NTBentmannC$ is regarded as an $\NTBentmannC$-right-module by the obvious action: The ideal $\NTBentmannnil\subset\NTBentmannC$ acts trivially, while $\NTBentmannss$ acts via right-multiplication in $\NTBentmannss\cong\ZBentmann^{\LCBentmann(X)^*}$. For an $\NTBentmannC$-module $M$, we set $\MssBentmann\defeqBentmann M/\NTBentmannnil\cdot M$.

For $Y\in\LCBentmann(X)^*$ we define the \emph{free $\NTBentmannC$-left-module on $Y$} by $P_Y(Z)\defeqBentmann\NTBentmann(Y,Z)$ for all $Z\in\LCBentmann(X)^*$ and similarly for morphisms $Z\to Z'$ in $\NTBentmannC$.
Analogously, we define the \emph{free $\NTBentmannC$-right-module on $Y$} by $Q_Y(Z)\defeqBentmann\NTBentmann(Z,Y)$ for all $Z\in\LCBentmann(X)^*$.
An $\NTBentmannC$-left/right-module is called \emph{free} if it is isomorphic to a direct
sum of degree-shifted free left/right-modules on objects $Y\in\LCBentmann(X)^*$.
It follows directly from Yoneda's Lemma that free $\NTBentmannC$-left/right-modules are projective.

An $\NTBentmann$-module $M$ is called \emph{exact} if the $\ZBentmann/2$-graded chain complexes
\[
\cdots\to M(U)\xrightarrow{i_U^Y} M(Y)\xrightarrow{r_Y^{Y\setminus U}}
M(Y\setminus U)\xrightarrow{\delta_{Y\setminus U}^U} M(U)[1]\to\cdots
\]
are exact for all $U,Y\in\LCBentmann(X)$ with~$U$ open in~$Y$.
An $\NTBentmannC$-module $M$ is called \emph{exact} if the corresponding
$\NTBentmann$-module is exact (see \cite{BKBentmann}).

We use the notation $C\ininBentmann\mathcal C$ to denote that~$C$ is an object in a category~$\mathcal C$.

In \cite{MN:FiltratedBentmann}, the functors $\FKBentmann_Y$ are shown to be representable, that is, there are objects $\RepBentmann_Y\ininBentmann\KKBentmanncat(X)$ and isomorphisms of functors $\FKBentmann_Y\cong\KKBentmann(X;\RepBentmann_Y,\blankBentmann)$. We let $\widehat{\FKBentmann}$ denote the stable \emph{co}homological functor on $\KKBentmanncat(X)$ represented by the same set of objects $\{\RepBentmann_Y\mid Y\in\LCBentmann(X)^*\}$; it takes values in $\NTBentmannC$-\emph{right}-modules. We warn that $\KKBentmann(X;A,\RepBentmann_Y)$ does not identify with the $\KBentmann$-homology of $A(Y)$. By Yoneda's lemma, we have $\FKBentmann(\RepBentmann_Y)\cong P_Y$ and $\widehat{\FKBentmann}(\RepBentmann_Y)\cong Q_Y$.

We occasionally use terminology from \cites{MN:HomologicalIBentmann,MN:FiltratedBentmann} concerning homological algebra in $\KKBentmanncat(X)$ relative to the ideal $\idealBentmann\defeqBentmann\ker(\FKBentmann)$ of morphisms in $\KKBentmanncat(X)$ inducing trivial module maps on $\FKBentmann$. An object $A\ininBentmann\KKBentmanncat(X)$ is called \emph{$\idealBentmann$-projective} if $\idealBentmann(A,B)=0$ for every $B\ininBentmann\KKBentmanncat(X)$. We recall from~\cite{MN:HomologicalIBentmann} that $\FKBentmann$ restricts to an equivalence of categories between the subcategories of $\idealBentmann$-projective objects in $\KKBentmanncat(X)$ and of projective objects in \(\ModBentmannc{\NTBentmannC(X)}\). Similarly, the functor~$\widehat{\FKBentmann}$ induces a contravariant equivalence between the $\idealBentmann$-projective objects in $\KKBentmanncat(X)$ and projective $\NTBentmannC$-\emph{right}-modules.

\section{Proof of Proposition \ref{pro:projdim1Bentmann}}
	\label{sec:projdim1Bentmann}

Recall the following result from \cite{MN:FiltratedBentmann}.

\begin{lemmaBentmann}[\cite{MN:FiltratedBentmann}*{Theorem 3.12}]
	\label{lem:projmodulesBentmann}
Let $X$ be a finite topological space. Assume that the ideal $\NTBentmannnil\subset\NTBentmannC(X)$ is nilpotent and that the decomposition $\NTBentmannC(X)=\NTBentmannnil\rtimes\NTBentmannss$ holds.
Let $M$ be an $\NTBentmannC(X)$-module. The following assertions are equivalent:
\begin{enumerate}[label=\textup{(\arabic*)}]
\item
$M$ is a free $\NTBentmannC(X)$-module.
\item
$M$ is a projective $\NTBentmannC(X)$-module.
\item
$\MssBentmann$ is a free Abelian group and $\TorBentmann_1^{\NTBentmannC(X)}(\NTBentmannss,M)=0$.
\end{enumerate}
\end{lemmaBentmann}

Now we prove Proposition~\textup{\ref{pro:projdim1Bentmann}}.
We consider the case $n=1$ first.
Choose an epimorphism $f\colon P\twoheadrightarrow M$ for some projective module $P$, and let $K$ be its kernel. $M$ has a projective resolution of length~1 if and only if $K$ is projective. By Lemma~\ref{lem:projmodulesBentmann}, this is equivalent to $K_\sss$ being a free Abelian group and $\TorBentmann_1^{\NTBentmannC}(\NTBentmannss,K)=0$. We have $\TorBentmann_1^{\NTBentmannC}(\NTBentmannss,K)=0$ if and only if $\TorBentmann_2^{\NTBentmannC}(\NTBentmannss,M)=0$ because these groups are isomorphic. We will show that $K_\sss$ is free if and only if $\TorBentmann_1^{\NTBentmannC}(\NTBentmannss,M)$ is free. The extension $K\rightarrowtail P\twoheadrightarrow M$ induces the following long exact sequence:
\begin{equation*}
	\label{eq:torsequenceBentmann}
0\to\TorBentmann_1^\NTBentmannC(\NTBentmannss,M)\to K_\sss\to P_\sss
\to\MssBentmann\to 0\;.
\end{equation*}
Assume that $K_\sss$ is free. Then its subgroup $\TorBentmann_1^{\NTBentmannC}(\NTBentmannss,M)$ is free as well. Conversely, if $\TorBentmann_1^{\NTBentmannC}(\NTBentmannss,M)$ is free, then $K_\sss$ is an extension of free Abelian groups and thus free. Notice that $P_\sss$ is free because $P$ is projective. The general case $n\in\NBentmann$ follows by induction using an argument based on syzygies as above. This completes the proof of Proposition~\textup{\ref{pro:projdim1Bentmann}}.

\section{Free Resolutions for \texorpdfstring{$\NTBentmannss$}{}}
	\label{sec:freeresBentmann}
The $\NTBentmannC$-right-module $\NTBentmannss$ decomposes as a direct sum $\bigoplus_{Y\in\LCBentmann(X)^*} S_Y$ of the simple submodules $S_Y$ which are given by $S_Y(Y)\cong\ZBentmann$ and $S_Y(Z)=0$ for $Z\neq Y$. We obtain
\[
 \TorBentmann_n^{\NTBentmannC}(\NTBentmannss,M)=\bigoplus_{Y\in\LCBentmann(X)^*} \TorBentmann_n^{\NTBentmann}(S_Y,M)\;.
\]
Our task is then to write down projective resolutions for the $\NTBentmannC$-right-modules $S_Y$. The first step is easy: we map $Q_Y$ onto $S_Y$ by mapping the class of the identity in $Q_Y(Y)$ to the generator of $S_Y(Y)$. Extended by zero, this yields an epimorphism $Q_Y\twoheadrightarrow S_Y$.

In order to surject onto the kernel of this epimorphism, we use the indecomposable transformations in $\NTBentmannC$ whose range is~$Y$. Denoting these by $\eta_i\colon W_i\to Y$, $1\leq i\leq n$, we obtain the two step resolution
\[
 \bigoplus_{i=1}^n Q_{W_i}\xrightarrow{\left(\begin{smallmatrix}\eta_1&\eta_2&\cdots&\eta_n\end{smallmatrix}\right)}Q_Y\twoheadrightarrow S_Y\;.
\]
In the notation of \cite{MN:FiltratedBentmann}, the map $\bigoplus_{i=1}^n Q_{W_i}\to Q_Y$ corresponds to a morphism $\phi\colon\RepBentmann_Y\to\bigoplus_{i=1}^n \RepBentmann_{W_i}$ of $\idealBentmann$-projectives in $\KKBentmanncat(X)$. If the mapping cone $C_\phi$ of~$\phi$ is again $\idealBentmann$-projective, the distinguished triangle $\Sigma C_\phi\to\RepBentmann_Y\xrightarrow{\phi}\bigoplus_{i=1}^n \RepBentmann_{W_i}\to C_\phi$ yields the projective resolution
\[
\cdots\to Q_Y\to Q_\phi[1]\to\bigoplus_{i=1}^n Q_{W_i}[1]\to Q_Y[1]\to Q_\phi\to\bigoplus_{i=1}^n Q_{W_i}\to Q_Y\twoheadrightarrow S_Y\;,
\]
where $Q_\phi=\FKBentmann(C_\phi)$. We denote periodic resolutions like this by
\[
 \xymatrix{Q_\phi\ar[r] & \bigoplus_{i=1}^n Q_{W_i}\ar[r] & Q_Y\to S_Y\ar@/_1pc/[ll]|\circ\;.}
\]
If the mapping cone $C_\phi$ is not $\idealBentmann$-projective, the situation has to be investigated individually. We will see examples of this in~\S\ref{sec:projdim2Bentmann} and~\S\ref{sec:projdim3Bentmann}. The resolutions we construct in these cases exhibit a certain six-term periodicity as well. However, they begin with a finite number of ``non-periodic steps'' (one in~\S\ref{sec:projdim2Bentmann} and two in~\S\ref{sec:projdim3Bentmann}), which can be considered as a symptom of the deficiency of the invariant filtrated $\KBentmann$-theory over non-accordion spaces from the homological viewpoint. We remark without proof that the mapping cone of the morphism $\phi\colon\RepBentmann_Y\to\bigoplus_{i=1}^n \RepBentmann_{W_i}$ is $\idealBentmann$-projective for every $Y\in\LCBentmann(X)^*$ if and only if~$X$ is a disjoint union of accordion spaces.

\section{Tensor Products with Free Right-Modules}
\label{sec:tensorproductsBentmann}
\begin{lemmaBentmann}
Let $M$ be an $\NTBentmannC$-left-module. There is an isomorphism $Q_Y\otimes_{\NTBentmannC} M\cong M(Y)$ of $\ZBentmann/2$-graded Abelian groups which is natural in $Y\ininBentmann\NTBentmannC$.
\end{lemmaBentmann}

\begin{proof}
This is a simple consequence of Yoneda's lemma and the tensor-hom adjunction.
\end{proof}

\begin{lemmaBentmann}
  \label{lem:tensor_productsBentmann}
Let 
$\Sigma\RepBentmann_{(3)}\xrightarrow{\gamma}\RepBentmann_{(1)}\xrightarrow{\alpha}\RepBentmann_{(2)}
\xrightarrow{\beta_*}\RepBentmann_{(3)}$
be a distinguished or anti-distinguished triangle in $\KKBentmanncat(X)$, where $\RepBentmann_{(i)}=\bigoplus_{j=1}^{m_i} \RepBentmann_{Y^i_j}\oplus\bigoplus_{k=1}^{n_i} \Sigma\RepBentmann_{Z^i_k}$ for $1\leq i\leq 3$, $m_i, n_i\in\NBentmann$ and $Y^i_j, Z^i_k\in\LCBentmann(X)^*$. Set $Q_{(i)}=\widehat{\FKBentmann}(\RepBentmann_{(i)})$. If $M=\FKBentmann(A)$ for some $A\ininBentmann\KKBentmanncat(X)$, then the induced sequence
\begin{equation}
	\label{eq:exactseqofmoduleentriesBentmann}
\begin{split}
\xymatrix{
Q_{(1)}\otimes_{\NTBentmannC} M\ar[r]^{\alpha^*\otimes\idBentmann_M} & Q_{(2)}\otimes_{\NTBentmannC} M\ar[r]^{\beta^*\otimes\idBentmann_M} & Q_{(3)}\otimes_{\NTBentmannC} M\ar[d]^{\gamma^*\otimes\idBentmann_M} \\
Q_{(3)}\otimes_{\NTBentmannC} M[1]\ar[u]^{\gamma^*\otimes\idBentmann_M[1]} & Q_{(2)}\otimes_{\NTBentmannC} M[1]\ar[l]^{\beta^*\otimes\idBentmann_M[1]} & Q_{(1)}\otimes_{\NTBentmannC} M[1]\ar[l]^{\alpha^*\otimes\idBentmann_M[1]}
}
\end{split}
\end{equation}
is exact.
\end{lemmaBentmann}

\begin{proof}
Using the previous lemma and the representability theorem, we naturally identify $Q_{(i)}\otimes_{\NTBentmannC} M\cong\KKBentmann(X;\RepBentmann_{(i)},A)$. Since, in triangulated categories, distinguished or anti-distinguished triangles induce long exact $\HomBentmann$-sequences, the sequence \eqref{eq:exactseqofmoduleentriesBentmann} is thus exact.
\end{proof}

\section{Proof of Proposition \ref{pro:projdim2Bentmann}}
	\label{sec:projdim2Bentmann}
We may restrict to connected $T_0$-spaces. In \cite{MN:BootstrapBentmann}, a list of isomorphism classes of connected $T_0$-spaces with three or four points is given. If $X$ is a disjoint union of accordion spaces, then the assertion follows from \cite{BKBentmann}. The remaining spaces fall into two classes:
\begin{enumerate}
 \item all connected non-accordion four-point $T_0$-spaces except for the pseudocircle;
 \item the pseudocircle (see~\S\ref{sec:description_of_pseudocircleBentmann}).
\end{enumerate}
The spaces in the first class have the following in common: If we fix two of them, say $X$, $Y$,
then there is an ungraded isomorphism $\Phi\colon\NTBentmannC(X)\to\NTBentmannC(Y)$ between the categories of natural transformations on the respective filtrated $\KBentmann$-theories such that the induced equivalence of ungraded module categories
\[
\Phi^*\colon\mathfrak{Mod}^\mathrm{ungr}\bigbBentmann{\NTBentmannC(Y)}_\textup{c}
\to\mathfrak{Mod}^\mathrm{ungr}\bigbBentmann{\NTBentmannC(X)}_\textup{c}
\]
restricts to a bijective correspondence between exact ungraded
$\NTBentmannC(Y)$-modules and exact ungraded $\NTBentmannC(X)$-modules.
Moreover, the isomorphism~$\Phi$ restricts to isomorphisms
from $\NTBentmannss(X)$ onto $\NTBentmannss(Y)$ and from
$\NTBentmannnil(X)$ onto $\NTBentmannnil(Y)$.
In particular, the assertion holds for $X$ if and only if it holds for $Y$.

The above is a consequence of the investigations in \cites{MN:FiltratedBentmann, BentmannBentmann, BKBentmann}; the same kind of relation was found in \cite{BKBentmann} for the categories of natural transformations associated to accordion spaces with the same number of points. As a consequence, it suffices to verify the assertion for one representative of the first class---we choose $Z_3$---and for the pseudocircle.

\subsection{Resolutions for the space \texorpdfstring{$Z_3$}{Z3}}
	\label{sec:resolutionsforZ3Bentmann}
We refer to \cite{MN:FiltratedBentmann} for a description of the category $\NTBentmannC(Z_3)$, which in particular implies, that the space $Z_3$ satisfies the conditions of Proposition~\ref{pro:projdim1Bentmann}. 
Using the extension triangles from \cite{MN:FiltratedBentmann}*{(2.5)}, the procedure described in~\S\ref{sec:freeresBentmann} yields the following projective resolutions induced by distinguished triangles as in Lemma~\ref{lem:tensor_productsBentmann}:
\begin{eqnarray*}
 &\xymatrix{Q_1[1]\ar[r] & Q_4\ar[r] & Q_{14}\to S_{14}\ar@/_1pc/[ll]|\circ\;,}\quad\textup{and similarly for $S_{24}$, $S_{34}$;}
\\
 &\xymatrix{Q_{1234}[1]\ar[r] & Q_{1}[1]\oplus Q_{2}[1]\oplus Q_{3}[1]\ar[r] & Q_{4}\to S_{4}\ar@/_1pc/[ll]|\circ\;;}
\\
 &\xymatrix{Q_{234}\ar[r] & Q_{1234}\ar[r] & Q_{1}\to S_{1}\ar@/_1pc/[ll]|\circ\;,}\quad\textup{and similarly for $S_{2}$, $S_{3}$.}
\end{eqnarray*}

Next we will deal with the modules $S_{jk4}$, where $1\leq j<k\leq 3$. We observe that there is a Mayer-Vietoris type exact sequence of the form
\begin{equation}
	\label{eq:jk4Bentmann}
\xymatrix{Q_{4}\ar[r] & Q_{j4}\oplus Q_{k4}\ar[r] & Q_{jk4}\ar@/_1pc/[ll]|\circ}.
\end{equation}

\begin{lemmaBentmann}
	\label{lem:induced_by_exact_triangleBentmann}
The candidate triangle $\Sigma\RepBentmann_{4}\to\RepBentmann_{jk4}\to\RepBentmann_{j4}\oplus\RepBentmann_{k4}\to\RepBentmann_4$
corresponding to the periodic part of the sequence \eqref{eq:jk4Bentmann} is distinguished or anti-distinguished \textup{(}depending on the choice of signs for the maps in \eqref{eq:jk4Bentmann}\textup{)}. 
\end{lemmaBentmann}

\begin{proof}
We give the proof for $j=1$ and $k=2$. The other cases follow from cyclicly permuting the indices $1$, $2$ and~$3$. We denote the morphism $\RepBentmann_{124}\to\RepBentmann_{14}\oplus\RepBentmann_{24}$ by~$\varphi$ and the corresponding map $Q_{14}\oplus Q_{24}\to Q_{124}$ in \eqref{eq:jk4Bentmann} by~$\varphi^*$. It suffices to check that $\widehat{\FKBentmann}(\ConeBentmann_\varphi)$ and $Q_4$ correspond, possibly up to a sign, to the same element in $\ExtBentmann^1_{\NTBentmannC(Z_3)^\opBentmann}\bigl(\kerBentmann(\varphi^*),\cokerBentmann(\varphi^*)[1]\bigr)$. We have $\cokerBentmann(\varphi^*)\cong S_{124}$ and an extension $S_{124}[1]\rightarrowtail Q_4\twoheadrightarrow\kerBentmann(\varphi^*)$. Since $\HomBentmann(Q_4,S_{124}[1])\cong S_{124}(4)[1]=0$ and $\ExtBentmann^1(Q_4,S_{124}[1])$ because~$Q_4$ is projective, the long exact $\ExtBentmann$-sequence yields $\ExtBentmann^1\bigl(\kerBentmann(\varphi^*),\cokerBentmann(\varphi^*)[1]\bigr)\cong\HomBentmann(S_{124}[1],S_{124}[1])\cong\ZBentmann$. Considering the sequence of transformations $3\xrightarrow{\delta} 124\xrightarrow{i} 1234\xrightarrow{r} 3$, it is straight-forward to check that such an extension corresponds to one of the generators $\pm 1\in\ZBentmann$ if and only if its underlying module is exact. This concludes the proof because both $\widehat{\FKBentmann}(\ConeBentmann_\varphi)$ and~$Q_4$ are exact.
\end{proof}

Hence we obtain the following projective resolutions induced by distinguished or anti-distinguished triangles as in Lemma~\ref{lem:tensor_productsBentmann}:
\begin{equation*}
\xymatrix{Q_{4}\ar[r] & Q_{j4}\oplus Q_{k4}\ar[r] & Q_{jk4}\to S_{jk4}\ar@/_1pc/[ll]|\circ}.
\end{equation*}
To summarize, by Lemma~\ref{lem:tensor_productsBentmann}, $\TorBentmann_n^{\NTBentmannC}(S_Y,M)=0$ for $Y\neq 1234$ and $n\geq 1$.

As we know from~\cite{MN:FiltratedBentmann}, the subset 1234 of $Z_3$ plays an exceptional role. In the notation of~\cite{MN:FiltratedBentmann} (with the direction of the arrows reversed because we are dealing with \emph{right}-modules), the kernel of the homomorphism $Q_{124}\oplus Q_{134}\oplus Q_{234}\xrightarrow{\left(\begin{smallmatrix}i&i&i\end{smallmatrix}\right)} Q_{1234}$ is of the form
\begin{equation*}
\begin{split}
\xymatrix{
&\ZBentmann\ar[ld]&0\ar[l]\ar[ld]&&\ZBentmann[1]\ar[ld]&\\
\ZBentmann^2&\ZBentmann\ar[l]&0\ar[ld]\ar[lu]&0\ar[l]\ar[ld]\ar[lu]&\ZBentmann[1]\ar[l]&\ZBentmann^2\;.\ar[l]|\circ\ar[ld]|\circ\ar[lu]|\circ\\
&\ZBentmann\ar[lu]&0\ar[l]\ar[lu]&&\ZBentmann[1]\ar[lu]&
}
\end{split}
\end{equation*}
It is the image of the module homomorphism
\begin{equation}
\label{eq:142434_to_124134234Bentmann}
Q_{14}\oplus Q_{24}\oplus Q_{34}\xrightarrow{\left(\begin{smallmatrix}i&-i&0\\-i&0&i\\0&i&-i\end{smallmatrix}\right)} Q_{124}\oplus Q_{134}\oplus Q_{234},
\end{equation}
the kernel of which, in turn, is of the form
\begin{equation*}
\begin{split}
\xymatrix{
&0\ar[ld]&\ZBentmann[1]\ar[l]\ar[ld]&&\ZBentmann[1]\ar[ld]&\\
\ZBentmann&0\ar[l]&\ZBentmann[1]\ar[ld]\ar[lu]&\ZBentmann[1]^3\ar[l]\ar[ld]\ar[lu]&\ZBentmann[1]\ar[l]&\ZBentmann\;.\ar[l]|\circ\ar[ld]|\circ\ar[lu]|\circ\\
&0\ar[lu]&\ZBentmann[1]\ar[l]\ar[lu]&&\ZBentmann[1]\ar[lu]&
}
\end{split}
\end{equation*}
A surjection from $Q_4\oplus Q_{1234}[1]$ onto this module is given by $\left(\begin{smallmatrix}i&i&i\\ \delta_{1234}^{14}&0&0\end{smallmatrix}\right)$, where $\delta_{1234}^{14}\defeqBentmann \delta_{3}^{14}\circ r_{1234}^{3}$. The kernel of this homomorphism has the form
\begin{equation*}
\begin{split}
\xymatrix{
&\ZBentmann[1]\ar[ld]&\ZBentmann[1]\ar[l]\ar[ld]&&0\ar[ld]&\\
\ZBentmann[1]&\ZBentmann[1]\ar[l]&\ZBentmann[1]\ar[ld]\ar[lu]&0\ar[l]\ar[ld]\ar[lu]&0\ar[l]&0\;.\ar[l]|\circ\ar[ld]|\circ\ar[lu]|\circ\\
&\ZBentmann[1]\ar[lu]&\ZBentmann[1]\ar[l]\ar[lu]&&0\ar[lu]&
}
\end{split}
\end{equation*}
This module is isomorphic to $\SyzBentmann_{1234}[1]$, where $\SyzBentmann_{1234}\defeqBentmann\ker(Q_{1234}\twoheadrightarrow S_{1234})$. Therefore, we end up with the projective resolution
\begin{equation}
	\label{eq:resolution_for_S1234Bentmann}
\resizebox{\textwidth}{!}{
\xymatrix @!R=0.1pc
{Q_{4}\oplus Q_{1234}[1]\ar[r] & Q_{14}\oplus Q_{24}\oplus Q_{34}\ar[r] & Q_{124}\oplus Q_{134}\oplus Q_{234}\ar[r]\ar@/_1pc/[ll]|\circ & Q_{1234}\to S_{1234}\;.
}
}
\end{equation}
The homomorphism from $Q_{124}\oplus Q_{134}\oplus Q_{234}$ to $Q_{4}\oplus Q_{1234}[1]$ is given by $\left(\begin{smallmatrix}0&0&-\delta_{234}^4\\ i&i&i \end{smallmatrix}\right)$, where $\delta_{234}^4\defeqBentmann \delta_2^4\circ r_{234}^2$.

\begin{lemmaBentmann}
	\label{lem:induced_by_exact_triangleBentmann2}
The candidate triangle in $\KKBentmanncat(X)$
corresponding to the periodic part of the sequence \eqref{eq:resolution_for_S1234Bentmann} is distinguished or anti-distinguished \textup{(}depending on the choice of signs for the maps in \eqref{eq:resolution_for_S1234Bentmann}\textup{)}. 
\end{lemmaBentmann}

\begin{proof}
The argument is analogous to the one in the proof of Lemma \ref{lem:induced_by_exact_triangleBentmann}. Again, we consider the group $\ExtBentmann^1_{\NTBentmannC(Z_3)^\opBentmann}\bigl(\kerBentmann(\varphi^*),\cokerBentmann(\varphi^*)[1]\bigr)$ where $\varphi^*$ now denotes the map \eqref{eq:142434_to_124134234Bentmann}. We have $\cokerBentmann(\varphi^*)\cong\SyzBentmann_{1234}$ and an extension $Q_{4}\rightarrowtail\kerBentmann(\varphi^*)\twoheadrightarrow S_{1234}[1]$. Using long exact sequences, we obtain
\begin{multline*}
\ExtBentmann^1\bigl(\kerBentmann(\varphi^*),\cokerBentmann(\varphi^*)[1]\bigr)
\cong\ExtBentmann^1(S_{1234}[1],\SyzBentmann_{1234}[1])\\
\cong\HomBentmann(S_{1234}[1],S_{1234}[1])\cong\ZBentmann.
\end{multline*}
Again, an extension corresponds to a generator if and only if its underlying module is exact.
\end{proof}

By the previous lemma and~\S\ref{sec:tensorproductsBentmann}, computing the tensor product of this complex with~$M$ and taking homology shows that $\TorBentmann_n^{\NTBentmannC}(\NTBentmannss,M)=0$ for $n\geq 2$ and that $\TorBentmann_1^{\NTBentmannC}(\NTBentmannss,M)$ is equal to $\TorBentmann_1^{\NTBentmannC}(S_{1234},M)$ and isomorphic to the homology of the complex \eqref{eq:homology_computing_Tor1Bentmann}.
\begin{exampleBentmann}
 For the filtrated $\KBentmann$-module with projective dimension~$2$ constructed in \cite{MN:FiltratedBentmann}*{\S5} we get $\TorBentmann_1^{\NTBentmannC}(\NTBentmannss,M)\cong\ZBentmann/k$.
\end{exampleBentmann}

\begin{remarkBentmann}
	\label{rem:Tor_for_SBentmann}
As explicated in the beginning of this section, the category $\NTBentmannC(S)$ corresponding to the four-point space~$S$ defined in the introduction is isomorphic in an appropriate sense to the category $\NTBentmannC(Z_3)$. As has been established in~\cite{BentmannBentmann}, the indecomposable morphisms in $\NTBentmannC(S)$ are organised in the diagram
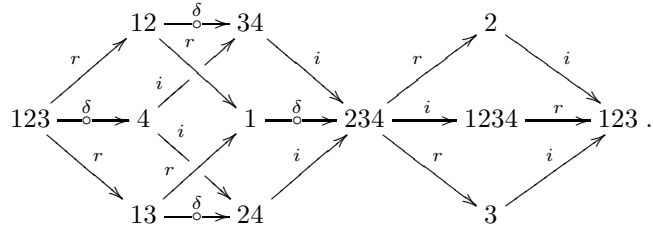
\begin{equation*}
\xymatrix{
										 & 12\ar[r]|\circ^\delta\ar[rd]^<<<<r    & 34\ar[rd]^i                    &                               & 2\ar[rd]^i   & \\
123\ar[ru]^r\ar[r]|\circ^\delta\ar[rd]^r & 4\ar[ru]|\hole^<<<i\ar[rd]|\hole^<<<i & 1\ar[r]|<<<<<\circ^<<<<<\delta & 234\ar[ru]^r\ar[r]^i\ar[rd]^r & 1234\ar[r]^r & 123\;. \\
										 & 13\ar[ru]^<<<<r\ar[r]|\circ^\delta    & 24\ar[ru]^i                    &                               & 3\ar[ru]^i   & }
\end{equation*}
In analogy to \eqref{eq:homology_computing_Tor1Bentmann}, we have that $\TorBentmann_1^{\NTBentmannC(S)}(\NTBentmannss,M)$ is isomorphic to the homology of the complex
\begin{multline}
  \label{eq:tor_sequence_for_SBentmann}
M(12)[1]\oplus M(4)\oplus M(13)[1]\xrightarrow{\left(\begin{smallmatrix}\delta&-r&0\\-i&0&i\\0&r&-\delta\end{smallmatrix}\right)} M(34)\oplus M(1)[1]\oplus M(24)\\
\xrightarrow{\left(\begin{smallmatrix}i&\delta&i\end{smallmatrix}\right)} M(234)\;,
\end{multline}
where $M=\FKBentmann(A)$ for some separable \CstarBentmann algebra~$A$ over~$X$.
\end{remarkBentmann}

\subsection{Resolutions for the pseudocircle}
  \label{sec:description_of_pseudocircleBentmann}
Let $C_2=\{1,2,3,4\}$ with the partial order defined by $1<3$, $1<4$, $2<3$, $2<4$. The topology on~$C_2$ is thus
given by $\{\emptyset,3,4,34,134,234,1234\}$.
Hence the non-empty, connected, locally closed subsets are
\[
\LCBentmann(C_2)^* = \{3,4,134,234,1234,13,14,23,24,124,123,1,2\}\;.
\]
The partial order on $ C_2$ corresponds to the directed graph
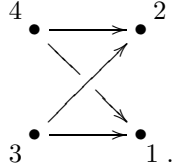
\[
\xy
(-9.5,9.5)*+{4}; (9.5,9.5)*+{2}; (-9.5,-9.5)*+{3}; (9.5,-9.5)*+{1\;.};
(-7,7)*+{\bullet}="p4"; (7,7)*+{\bullet}="p2";
(-7,-7)*+{\bullet}="p3"; (7,-7)*+{\bullet}="p1";
{\ar "p4"; "p2"}; {\ar|\hole "p4"; "p1"}; {\ar "p3"; "p2"}; {\ar "p3"; "p1"};
\endxy
\]
The space $C_2$ is the only $T_0$-space with at most four points with the property that its order complex (see \cite{MN:FiltratedBentmann}*{Definition 2.6}) is not contractible; in fact, it is homeomorphic to the circle $\SphereBentmann^1$. Therefore, by the representability theorem \cite{MN:FiltratedBentmann}*{\S 2.1} we find
\[
\NTBentmann_*(C_2,C_2)\cong\KKBentmann_*(X;\RepBentmann_{C_2},\RepBentmann_{C_2})\cong\KBentmann_*\bigl(\RepBentmann_{C_2}(C_2)\bigr)
\cong\KBentmann^*\left(\SphereBentmann^1\right)\cong\ZBentmann\oplus\ZBentmann[1]\;,
\]
that is, there are non-trivial odd natural transformations
$\FKBentmann_{C_2}\Rightarrow\FKBentmann_{C_2}$. These are generated, for instance,
by the composition $C_2\xrightarrow{r} 1\xrightarrow{\delta} 3\xrightarrow{i}C_2$.
This follows from the description of the category $\NTBentmannC(C_2)$ below.
Note that $\delta_{C_2}^{C_2}\circ \delta_{C_2}^{C_2}$
vanishes because it factors through $r_{13}^1\circ i_3^{13}=0$.

Figure~\ref{fig:indecomposablesC2Bentmann} displays a set of indecomposable transformations generating the category $\NTBentmannC(C_2)$ determined in~\cite{BentmannBentmann}*{\S 6.3.2}, where also a list of relations generating the relations in the category $\NTBentmannC(C_2)$ can be found. From this, it is straight-forward to verify that the space $C_2$ satisfies the conditions of Proposition~\ref{pro:projdim1Bentmann}.

\begin{figure}[htbp]
\begin{equation*}
\begin{split}
\xymatrix{
                               &                                       & 13\ar@/^/[rd]^i                                                  &                                  &                                                            &   & \\
3\ar[r]^i\ar@/^/[rdd]|(0.54)\hole^<<<<<i & 134\ar@/^/[ru]^r\ar[r]^r\ar@/_/[rd]^i & 14\ar@/^/[rdd]|(0.54)\hole_<<i                                           & 123\ar[r]^r\ar@/^/[rdd]|(0.54)\hole^<<<<<r & 1\ar[r]|\circ^\delta\ar@/^/[rdd]|<<<<<<\circ|(0.54)\hole^<<<<<\delta & 3 & \\
                               &                                       & 1234\ar@/^/[ru]|(0.42)\hole^>>>r\ar@/_/[rd]|(0.42)\hole_>>>r    &                                  &                                                            &   & \\
4\ar[r]^i\ar@/_/[ruu]^<<<<<i      & 234\ar@/^/[ru]^i\ar[r]^r\ar@/_/[rd]^r & 23\ar@/_/[ruu]^<<i                                                 & 124\ar[r]^r\ar@/_/[ruu]^<<<<<r      & 2\ar[r]|\circ^\delta\ar@/_/[ruu]|<<<<<<\circ^<<<<<<\delta      & 4 & \\
                               &                                       & 24\ar@/_/[ru]^i                                                  &                                  &                                                            &   & }
\end{split}
\end{equation*}
\caption{Indecomposable natural transformations in $\NTBentmannC(C_2)$}
\label{fig:indecomposablesC2Bentmann}
\end{figure}

Proceeding as described in~\S\ref{sec:freeresBentmann}, we find projective resolutions of the following form (we omit explicit descriptions of the boundary maps):
\begin{eqnarray*}
 &\xymatrix{Q_{123}[1]\ar[r] & Q_1[1]\oplus Q_2[1]\ar[r] & Q_{3}\to\ar@/_1pc/[ll]|\circ S_{3}\;,}
\quad\textup{and similarly for $S_{4}$;}
\\
 &\xymatrix{Q_1[1]\ar[r] & Q_3\oplus Q_4\ar[r] & Q_{134}\to\ar@/_1pc/[ll]|\circ S_{134}\;,}
\quad\textup{and similarly for $S_{234}$;}
\\
 &\xymatrix{Q_{4}\ar[r] & Q_{134}\ar[r] & Q_{13}\to\ar@/_1pc/[ll]|\circ S_{13}\;,}
\quad\textup{and similarly for $S_{14}$, $S_{23}$, $S_{24}$;}
\\
 &\xymatrix{Q_{3}\oplus Q_{4}\ar[r] & Q_{134}\oplus Q_{234}\ar[r] & Q_{1234}\to\ar@/_1pc/[ll]|\circ  S_{1234}\;;}
\\
 &\xymatrix{Q_{4}\oplus Q_{123}[1]\ar[r] & Q_{134}\oplus Q_{234}\ar[r] & Q_{1234}\oplus Q_{13}\oplus Q_{23}\to Q_{123}\to S_{123}\ar@/_1pc/[ll]|\circ\;,}
\end{eqnarray*}
and similarly for $S_{124}$;
\[
 \xymatrix{Q_{234}\oplus Q_{1}[1]\ar[r] & Q_{1234}\oplus Q_{23}\oplus Q_{24}\ar[r] & Q_{123}\oplus Q_{124}\to Q_{1} \to S_{1}\ar@/_1pc/[ll]|\circ\;,}
\]
and similarly for $S_{2}$. Again, the periodic part of each of these resolutions is induced by an extension triangle, a Mayer-Vietoris triangle as in Lemma~\ref{lem:induced_by_exact_triangleBentmann} or a more exotic (anti-)distinguished triangle as in Lemma~\ref{lem:induced_by_exact_triangleBentmann2} (we omit the analogous computation here).

We get $\TorBentmann_1^{\NTBentmannC}(S_Y,M)=0$ for every $Y\in\LCBentmann(C_2)^*\setminus\{123,124,1,2\}$, and $\TorBentmann_n^{\NTBentmannC}(S_Y,M)=0$ for all $Y\in\LCBentmann(C_2)^*$ and $n\geq 2$. Therefore,
\[
 \TorBentmann_1^{\NTBentmannC}(\NTBentmannss,M)\cong\bigoplus_{Y\in\{123,124,1,2\}}\TorBentmann_1^{\NTBentmannC}(S_Y,M)\;.
\]
The four groups $\TorBentmann_1^{\NTBentmannC}(S_Y,M)$ with $Y\in\{123,124,1,2\}$ can be described explicitly as in~\S\ref{sec:resolutionsforZ3Bentmann} using the above resolutions. This finishes the proof of Proposition~\textup{\ref{pro:projdim2Bentmann}}.

\section{Proof of Proposition \ref{pro:rationalBentmann}}
	\label{sec:rationalBentmann}
We apply the Meyer-Nest machinery to the homological functor $\FKBentmann\otimes\QBentmann$ on the triangulated category $\KKBentmanncat(X)\otimes\QBentmann$. We need to show that every $\NTBentmannC\otimes\QBentmann$ module of the form $M=\FKBentmann(A)\otimes\QBentmann$ has a projective resolution of length~$1$.
It is easy to see that analogues of Propositions~\ref{pro:projdim1Bentmann} and~\ref{pro:projdim2Bentmann} hold. In particular, the term $\TorBentmann_2^{\NTBentmannC\otimes\QBentmann}(\NTBentmannss\otimes\QBentmann,M)$ always vanishes. Here we use that $\QBentmann$ is a flat $\ZBentmann$-module, so that tensoring with $\QBentmann$ turns projective $\NTBentmannC$-module resolutions into projective $\NTBentmannC\otimes\QBentmann$-module resolutions. Moreover, the freeness condition for the $\QBentmann$-module $\TorBentmann_1^{\NTBentmannC\otimes\QBentmann}(\NTBentmannss\otimes\QBentmann,M)$ is empty since~$\QBentmann$ is a field.

\section{Proof of Proposition \ref{pro:projdim3Bentmann}}
	\label{sec:projdim3Bentmann}
The computations to determine the category $\NTBentmannC(Z_4)$ are very similar to those for the category $\NTBentmannC(Z_3)$ which were carried out in \cite{MN:FiltratedBentmann}. We summarise its structure in Figure~\ref{fig:indecomposablesZ4Bentmann}.
The relations in $\NTBentmannC(Z_4)$ are generated by the following:
\begin{itemize}
 \item the hypercube with vertices $5,15,25,\ldots,12345$ is a commuting diagram;
\item the following compositions vanish:
\begin{gather*} 1235\xrightarrow{i} 12345\xrightarrow{r} 4\;, \quad 1245\xrightarrow{i} 12345\xrightarrow{r} 3\;,\\ 1345\xrightarrow{i} 12345\xrightarrow{r} 2\;,\quad 2345\xrightarrow{i} 12345\xrightarrow{r} 1\;,\\
 1\xrightarrow{\delta} 5\xrightarrow{i} 15\;,\quad 2\xrightarrow{\delta} 5\xrightarrow{i} 25\;,\quad 3\xrightarrow{\delta} 5\xrightarrow{i} 35\;,\quad 4\xrightarrow{\delta} 5\xrightarrow{i} 45\;;
\end{gather*}
\item the sum of the four maps $12345\to 5$ via $1$, $2$, $3$, and~$4$ vanishes.
\end{itemize}
This implies that the space $Z_4$ satisfies the conditions of Proposition~\ref{pro:projdim1Bentmann}.

\begin{figure}[htbp]
\begin{equation*}
\begin{split}
\resizebox{!}{5pt}{
\xymatrix{
&&125\ar[rd]^i\ar[rdd]^i&&&&\\
&15\ar[ru]^i\ar[r]^i\ar[rd]^i&135\ar[r]^i\ar[rddd]^i&1235\ar[rdd]^i&&1\ar[rdd]|\circ^\delta&\\
&25\ar[ruu]^i\ar[rdd]^i\ar[rddd]^i&145\ar[r]^i\ar[rdd]^i&1245\ar[rd]^i&&2\ar[rd]|\circ_\delta&\\
5\ar[ruu]^i\ar[ru]^i\ar[rd]^i\ar[rdd]^i&&&&12345\ar[ruu]^r\ar[ru]^r\ar[rd]^r\ar[rdd]^r&&5\\
&35\ar[ruuu]^i\ar[r]^i\ar[rdd]^i&235\ar[ruuu]^i\ar[rd]^i&1345\ar[ru]^i&&3\ar[ru]|\circ^\delta&\\
&45\ar[ruuu]^i\ar[r]^i\ar[rd]^i&245\ar[ruuu]^i\ar[r]^i&2345\ar[ruu]^i&&4\ar[ruu]|\circ^\delta&\\
&&345\ar[ruu]^i\ar[ru]^i&&&&
}
}
\end{split}
\end{equation*}
\caption{Indecomposable natural transformations in $\NTBentmannC(Z_4)$}
\label{fig:indecomposablesZ4Bentmann}
\end{figure}

In the following, we will define an exact $\NTBentmannC$-left-module $M$ and compute $\TorBentmann_2^{\NTBentmannC}(S_{12345},M)$. By explicit computation, one finds a projective resolution of the simple $\NTBentmannC$-right-module $S_{12345}$ of the following form (again omitting explicit formulas for the boundary maps):
\[
\begin{tikzpicture}[remember picture=true, descr/.style={fill=white,inner sep=0pt}]
        \matrix (m) [
            matrix of math nodes,
             row sep=1em,
             column sep=2em,
              text height=1.5ex, text depth=0.25ex
        ]
        {
	 &  &  \\
	Q_{5}\oplus\alternativebigoplusBentmann_{\mathclap{1\leq i\leq 4}} Q_{12345\setminus i}[1] &
        \alternativebigoplusBentmann_{\mathclap{1\leq l\leq 4}} Q_{l5}\oplus Q_{12345}[1] &
        \alternativebigoplusBentmann_{\mathclap{1\leq j<k\leq 4}} Q_{jk5} \\
        &&\\
          \alternativebigoplusBentmann_{\mathclap{1\leq i\leq 4}} Q_{12345\setminus i} &
        Q_{12345} &
        S_{12345}. \\
        };
        \path[overlay,->,font=\scriptsize]
        (m-2-1) edge (m-2-2)
        (m-2-2) edge (m-2-3)
	(m-2-3) edge[out=10,in=175] node[descr] {$\circ$} (m-2-1)
        (m-2-3) edge[out=346,in=173] (m-4-1)
        (m-4-1) edge (m-4-2)
        (m-4-2) edge (m-4-3);
\end{tikzpicture}
\]
Notice that this sequence is periodic as a cyclic six-term sequence except for the first \emph{two} steps.

Consider the exact $\NTBentmannC$-left-module $M$ defined by the exact sequence
\begin{equation}
\begin{split}
	\label{eq:length2resolutionBentmann}
 0\to P_{12345}\xrightarrow{\left(\begin{smallmatrix}i\\i\\i\\i\end{smallmatrix}\right)}
\bigoplus_{1\leq i\leq 4} P_{12345\setminus i}
\xrightarrow{\left(\begin{smallmatrix}i&-i&0&0\\-i&0&i&0\\0&i&-i&0\\i&0&0&-i\\0&-i&0&i\\0&0&i&-i\end{smallmatrix}\right)}
\bigoplus_{1\leq j<k\leq 4} P_{jk5}\epiBentmann M\;.
\end{split}
\end{equation}
We have $\bigoplus_{1\leq l\leq 4} M(l5)\oplus M(12345)[1]\cong 0\oplus \ZBentmann^3$,
$\bigoplus_{1\leq j<k\leq 4} M(jk5)\cong \ZBentmann^6$, and
$M(5)\oplus\bigoplus_{1\leq i\leq 4} M(12345\setminus i)[1]\cong \ZBentmann[1]\oplus \ZBentmann[1]^8$.
Since
\[
\xymatrix{
\alternativebigoplusBentmann_{1\leq l\leq 4} M(l5)\oplus M(12345)[1]\ar[r] & \alternativebigoplusBentmann_{1\leq j<k\leq 4} M(jk5)\ar[d]|\circ \\
& M(5)\oplus\alternativebigoplusBentmann_{1\leq i\leq 4} M(12345\setminus i)[1]\ar[lu]
}
\]
is exact, a rank argument shows that the map
\[
\alternativebigoplusBentmann_{1\leq l\leq 4} M(l5)\oplus M(12345)[1]\to \alternativebigoplusBentmann_{1\leq j<k\leq 4} M(jk5)
\]
is zero. On the other hand, the kernel of the map
\[
\alternativebigoplusBentmann_{1\leq j<k\leq 4} M(jk5)
\xrightarrow{\left(\begin{smallmatrix}i&-i&0&i&0&0\\-i&0&i&0&-i&0\\0&i&-i&0&0&i\\0&0&0&-i&i&-i\end{smallmatrix}\right)}
\alternativebigoplusBentmann_{1\leq i\leq 4} M(12345\setminus i)
\]
is non-trivial; it consists precisely of the elements in
\[
\alternativebigoplusBentmann_{1\leq j<k\leq 4} M(jk5)\cong\alternativebigoplusBentmann_{1\leq j<k\leq 4}\ZBentmann[\idBentmann_{jk5}^{jk5}]
\]
which are multiples of $([\idBentmann_{jk5}^{jk5}])_{1\leq j<k\leq 4}$. This shows $\TorBentmann_2^{\NTBentmannC}(S_{12345},M)\cong\ZBentmann$. Hence, by Proposition~\ref{pro:projdim1Bentmann}, the module~$M$ has projective dimension at least~2. On the other hand, \eqref{eq:length2resolutionBentmann} is a resolution of length~$2$. Therefore, the projective dimension of~$M$ is exactly~$2$.

Let $k\in\NBentmann_{\geq 2}$ and define $M_k=M\otimes_\ZBentmann \ZBentmann/k$. Since $\TorBentmann_2^{\NTBentmannC}(S_{12345},M_k)\cong\ZBentmann/k$ is non-free, Proposition~\ref{pro:projdim1Bentmann} shows that $M_k$ has at least projective dimension~3. On the other hand, if we abbreviate the resolution \eqref{eq:length2resolutionBentmann} for $M$ by
\begin{equation}
 	\label{eq:proj_res_of _entry-free_moduleBentmann}
0\to P^{(5)}\xrightarrow{\alpha}P^{(4)}\xrightarrow{\beta}P^{(3)}\twoheadrightarrow M\;,
\end{equation}
a projective resolution of length~3 for~$M_k$ is given by
\[
0\to P^{(5)}\xrightarrow{\left(\begin{smallmatrix}k\\ \alpha\end{smallmatrix}\right)}
P^{(5)}\oplus P^{(4)}\xrightarrow{\left(\begin{smallmatrix}\alpha&-k\\0&\beta\end{smallmatrix}\right)}
P^{(4)}\oplus P^{(3)}\xrightarrow{\left(\begin{smallmatrix}\beta&k\end{smallmatrix}\right)}
P^{(3)}\twoheadrightarrow M_k\;,
\]
where $k$ denotes multiplication by $k$.

It remains to show that the modules~$M$ and~$M_k$ can be realised as the filtrated $\KBentmann$-theory of objects in $\BootBentmann(X)$. It suffices to prove this for the module~$M$ since tensoring with the Cuntz algebra $\CuntzBentmann_{k+1}$ then yields a separable \CstarBentmann algebra with filtrated $\KBentmann$-theory~$M_k$ by the K\"unneth Theorem.

The projective resolution \eqref{eq:proj_res_of _entry-free_moduleBentmann} can be written as
\begin{equation*}
0\to \FKBentmann(P^2)\xrightarrow{\FKBentmann(f_2)}\FKBentmann(P^1)\xrightarrow{\FKBentmann(f_1)}\FKBentmann(P^0)\twoheadrightarrow M,
\end{equation*}
because of the equivalence of the category of projective $\NTBentmannC$-modules and the category of $\idealBentmann$-projective objects in $\KKBentmanncat(X)$. Let $N$ be the cokernel of the module map $\FKBentmann(f_2)$. Using \cite{MN:FiltratedBentmann}*{Theorem 4.11}, we obtain an object $A\ininBentmann\BootBentmann(X)$ with $\FKBentmann(A)\cong N$. We thus have a commutative diagram of the form
\begin{equation*}
\xymatrix{
0\ar[r] & \FKBentmann(P^2)\ar[r]^{\FKBentmann(f_2)} & \FKBentmann(P^1)\ar[rr]^{\FKBentmann(f_1)}\ar@{->>}[rd] & & \FKBentmann(P^0)\ar@{->>}[r] & M\;.\\
&&& \FKBentmann(A)\ar@{ >Bentmann->}[ur]^\gamma &&
}
\end{equation*}
Since $A$ belongs to the bootstrap class $\BootBentmann(X)$ and $\FKBentmann(A)$ has a projective resolution of length~1, we can apply the universal coefficient theorem to lift the homomorphism~$\gamma$ to an element $f\in\KKBentmann(X;A,P^0)$. Now we can argue as in the proof of \cite{MN:FiltratedBentmann}*{Theorem 4.11}: since~$f$ is $\idealBentmann$-monic, the filtrated $\KBentmann$-theory of its mapping cone is isomorphic to $\cokerBentmann(\gamma)\cong M$. This completes the proof of Proposition~\textup{\ref{pro:projdim3Bentmann}}.

\section{Cuntz-Krieger Algebras with Projective Dimension~2}
	\label{sec:Cuntz-KriegerBentmann}

In this section we exhibit a \index{Cuntz-Krieger algebra} Cuntz-Krieger algebra~$A$ which is a tight \CstarBentmann{}algebra over the space~$Z_3$ and for which the odd part of $\TorBentmann_1^{\NTBentmannC(Z_3)}\bigbBentmann{\NTBentmannss,\FKBentmann(A)}$---denoted $\TorBentmann_1^\mathrm{odd}$ in the following---is not free. By Proposition~\ref{pro:projdim2Bentmann} this \CstarBentmann algebra has projective dimension~2 in filtrated $\KBentmann$-theory.

In the following we will adhere to the conventions for graph algebras and adjacency matrices from~\cite{CarlsenEilersTomfordeBentmann}.
Let $E$ be the finite graph with vertex set $E^0=\{v_1,v_2,\ldots,v_8\}$ and edges corresponding to the adjacency matrix
\begin{equation}
 	\label{eq:CK-matrixBentmann}
 \begin{pmatrix}
  	B_4&0&0&0\\
	X_1&B_1&0&0\\
	X_2&0&B_2&0\\
	X_3&0&0&B_3
 \end{pmatrix}
\defeqBentmann
 \begin{pmatrix}
\begin{pmatrix}3&2\\2&3\end{pmatrix}&
0&
0&
0\\
\begin{pmatrix}1&1\\1&1\end{pmatrix}&
\begin{pmatrix}3&2\\1&2\end{pmatrix}&
0&0\\
\begin{pmatrix}1&1\\1&1\end{pmatrix}
&0&
\begin{pmatrix}3&2\\1&2\end{pmatrix}&
0\\
\begin{pmatrix}1&1\\1&1\end{pmatrix}
&0&0&
\begin{pmatrix}3&2\\1&2\end{pmatrix}
 \end{pmatrix}\;.
\end{equation}
Since this is a finite graph with no sinks and no sources, the associated graph \CstarBentmann algebra $C^*(E)$ is in fact a Cuntz-Krieger algebra (we can replace~$E$ with its \emph{edge graph}; see \cite{RaeburnBentmann}*{Remark 2.8}). Moreover, the graph~$E$ is easily seen to fulfill condition~(K) because every vertex is the base of two or more simple cycles. As a consequence, the adjacency matrix of the edge graph of~$E$ fulfills condition~(II) from~\cite{CuntzBentmann}. In fact, condition~(K) is designed as a generalisation of condition~(II): see, for instance,~\cite{KumjianBentmann}.

Applying \cite{RaeburnBentmann}*{Theorem 4.9}---and carefully translating between different graph algebra conventions---we find that the ideals of $C^*(E)$ correspond bijectively and in an inclusion-preserving manner to the open subsets of the space~$Z_3$. By \cite{MN:BootstrapBentmann}*{Lemma 2.35}, we may turn~$A$ into a tight \CstarBentmann algebra over~$Z_3$ by declaring $A(\{4\})=I_{\{v_1,v_2\}}$, $A(\{1,4\})=I_{\{v_1,v_2,v_3,v_4\}}$, $A(\{2,4\})=I_{\{v_1,v_2,v_5,v_6\}}$ and $A(\{3,4\})=I_{\{v_1,v_2,v_7,v_8\}}$, where~$I_S$ denotes the ideal corresponding to the saturated hereditary subset~$S$.

It is known how to compute the six-term sequence in $\KBentmann$-theory for an extension of graph \CstarBentmann algebras: see~\cite{CarlsenEilersTomfordeBentmann}. Using this and Proposition~\ref{pro:projdim2Bentmann}, $\TorBentmann_1^\mathrm{odd}$ is the homology of the complex
\begin{equation}
   \label{eq:CK-quotientBentmann}
\ker(\phi_0)
\xrightarrow{\left(\begin{smallmatrix}i&-i&0\\-i&0&i\\0&i&-i\end{smallmatrix}\right)}
\ker(\phi_1)
\xrightarrow{\left(\begin{smallmatrix}i&i&i\end{smallmatrix}\right)}
\ker(\phi_2)\;,
\end{equation}
\[
\textup{where}\quad
\phi_0=\diagBentmann\left(
\left(\begin{smallmatrix}B'_4&X_1^t\\0&B'_1\end{smallmatrix}\right),
\left(\begin{smallmatrix}B'_4&X_2^t\\0&B'_2\end{smallmatrix}\right),
\left(\begin{smallmatrix}B'_4&X_3^t\\0&B'_3\end{smallmatrix}\right)\right)\;,\quad
\phi_2=\left(\begin{smallmatrix}B'_4&X_1^t&X_2^t&X_3^t\\0&B'_1&0&0\\0&0&B'_2&0\\0&0&0&B'_3\end{smallmatrix}\right)\;,
\]
\[
\phi_1=\diagBentmann\left(
\left(\begin{smallmatrix}B'_4&X_1^t&X_2^t\\0&B'_1&0\\0&0&B'_2\end{smallmatrix}\right),
\left(\begin{smallmatrix}B'_4&X_1^t&X_3^t\\0&B'_1&0\\0&0&B'_3\end{smallmatrix}\right),
\left(\begin{smallmatrix}B'_4&X_2^t&X_3^t\\0&B'_2&0\\0&0&B'_3\end{smallmatrix}\right)\right)\;,
\]
and $B'_4=B_4^t-\left(\begin{smallmatrix}1&0\\0&1\end{smallmatrix}\right)
=\left(\begin{smallmatrix}2&2\\2&2\end{smallmatrix}\right)$ and $B'_j=B_j^t-\left(\begin{smallmatrix}1&0\\0&1\end{smallmatrix}\right)
=\left(\begin{smallmatrix}2&1\\2&1\end{smallmatrix}\right)$ for $1\leq j\leq 3$. We obtain a commutative diagram
\begin{equation}
 	\label{eq:commutative _diagramBentmann}
\begin{split}
 \xymatrix{
\ker(\phi_0)\ar[d]^{f_K}\ar@{ >Bentmann->}[r] & (\ZBentmann^{\oplus 2})^{\oplus (2\cdot 3)}\ar[d]^{f}\ar@{->>}[r]^{\phi_0} &\imBentmann(\phi_0)\ar[d]^{f_I} \\
\ker(\phi_1)\ar[d]^{g_K}\ar@{ >Bentmann->}[r] & (\ZBentmann^{\oplus 2})^{\oplus (3\cdot 3)}\ar[d]^{g}\ar@{->>}[r]^{\phi_1} &\imBentmann(\phi_1)\ar[d]^{g_I} \\
\ker(\phi_2)\ar@{ >Bentmann->}[r] & (\ZBentmann^{\oplus 2})^{\oplus (4\cdot 1)}\ar@{->>}[r]^{\phi_2} &\imBentmann(\phi_2)\;,
}
\end{split}
\end{equation}
where $f$ and $g$ have the block forms
\[
 f=
\left(\begin{smallmatrix}
\idBentmann&0&-\idBentmann&0&0&0\\
0&\idBentmann&0&0&0&0\\
0&0&0&-\idBentmann&0&0\\
-\idBentmann&0&0&0&\idBentmann&0\\
0&-\idBentmann&0&0&0&0\\
0&0&0&0&0&\idBentmann\\
0&0&\idBentmann&0&-\idBentmann&0\\
0&0&0&\idBentmann&0&0\\
0&0&0&0&0&-\idBentmann
\end{smallmatrix}\right)\;,
\qquad
 g=
\left(\begin{smallmatrix}
\idBentmann&0&0&\idBentmann&0&0&\idBentmann&0&0\\
0&\idBentmann&0&0&\idBentmann&0&0&0&0\\
0&0&\idBentmann&0&0&0&0&\idBentmann&0\\
0&0&0&0&0&\idBentmann&0&0&\idBentmann
\end{smallmatrix}\right)\;,
\]
and $f_K\defeqBentmann f|_{\ker(\phi_0)}$, $f_I\defeqBentmann f|_{\imBentmann(\phi_0)}$, $g_K\defeqBentmann g|_{\ker(\phi_1)}$, $g_I\defeqBentmann g|_{\imBentmann(\phi_1)}$. Notice that~$f$ and~$g$ are defined in a way such that the restrictions $f|_{\ker(\phi_0)}$ and $g|_{\ker(\phi_1)}$ are exactly the maps from \eqref{eq:CK-quotientBentmann} in the identification made above.

We abbreviate the above short exact sequence of cochain complexes \eqref{eq:commutative _diagramBentmann} as
$K_\bullet\rightarrowtail Z_\bullet\twoheadrightarrow I_\bullet$.
The part $\HomologyBentmann^0(Z_\bullet)\to\HomologyBentmann^0(I_\bullet)\to\HomologyBentmann^1(K_\bullet)\to\HomologyBentmann^1(Z_\bullet)$ in the corresponding long exact homology sequence can be identified with
\[
 \ker(f)\xrightarrow{\phi_0}\ker(f_I)\to\frac{\ker(g_K)}{\imBentmann(f_K)}\to 0\;.
\]
Hence
\[
\TorBentmann_1^\mathrm{odd}\cong\frac{\ker(g_K)}{\imBentmann(f_K)}\cong\frac{\ker(f_I)}{\phi_0\bigbBentmann{\ker(f)}}\cong\frac{\ker(f)\cap\imBentmann(\phi_0)}{\phi_0\bigbBentmann{\ker(f)}}\;.
\]
We have $\ker(f)=\{(v,0,v,0,v,0)\mid v\in\ZBentmann^2\}\subset (\ZBentmann^{\oplus 2})^{\oplus (2\cdot 3)}$.

From the concrete form \eqref{eq:CK-matrixBentmann} of the adjacency matrix, we find that $\ker(f)\cap\imBentmann(\phi_0)$ is the free cyclic group generated by $(1,1,0,0,1,1,0,0,1,1,0,0)$, while $\phi_0\bigbBentmann{\ker(f)}$ is the subgroup generated by $(2,2,0,0,2,2,0,0,2,2,0,0)$. Hence $\TorBentmann_1^\mathrm{odd}\cong\ZBentmann/2$ is not free.

Now we briefly indicate how to construct a similar counterexample for the space~$S$.
Consider the integer matrix
\begin{equation*}
 	\label{eq:CK-matrixSBentmann}
 \begin{pmatrix}
  	B_4&0&0&0\\
	X_{43}&B_3&0&0\\
	X_{42}&0&B_2&0\\
	X_{41}&X_{31}&X_{21}&B_1
 \end{pmatrix}
\defeqBentmann
\begin{pmatrix}
  	\begin{pmatrix}3\end{pmatrix}&0&0&0\\
	\begin{pmatrix}2\end{pmatrix}&\begin{pmatrix}3\end{pmatrix}&0&0\\
	\begin{pmatrix}2\end{pmatrix}&0&\begin{pmatrix}3\end{pmatrix}&0\\
	\begin{pmatrix}2\\0\end{pmatrix}&\begin{pmatrix}1\\0\end{pmatrix}&\begin{pmatrix}1\\0\end{pmatrix}&\begin{pmatrix}2&1\\1&2\end{pmatrix}
 \end{pmatrix}\;.
\end{equation*}
The corresponding graph $F$ fulfills condition~(K) and has no sources or sinks. The associated graph \CstarBentmann algebra $C^*(F)$ is therefore a Cuntz-Krieger algebra satisfying condition~(II). It is easily read from the block structure of the edge matrix that the primitive ideal space of $C^*(F)$ is homeomorphic to~$S$. We are going to compute the even part of $\TorBentmann_1^{\NTBentmannC(S)}\bigbBentmann{\NTBentmannss,\FKBentmann(C^*(F))}$. Since the nice computation methods from the previous example do not carry over, we carry out a more ad hoc calculation.

By Remark~\ref{rem:Tor_for_SBentmann}, the even part of our $\TorBentmann$-term is isomorphic to the homology of the complex
\begin{equation*}
\xymatrix{
 \ker\left({\begin{smallmatrix}B'_2&X_{21}^t\\0&B'_1\end{smallmatrix}}\right)\ar[r]|\circ^*+{\left(\begin{smallmatrix}X_{42}^t&X_{41}^t\\0&X_{31}^t\end{smallmatrix}\right)}\ar[rd]^<<<<<{-r}    & \cokerBentmann\left({\begin{smallmatrix}B'_4&X_{43}^t\\0&B'_3\end{smallmatrix}}\right)\ar[rd]^i                    &         \\
  \cokerBentmann(B'_4)\ar[ru]|\hole^<<<<<{-i}\ar[rd]|\hole^<<<<<i & \ker(B'_1)\ar[r]|<<<<<\circ^<<<<<{\left(\begin{smallmatrix}X_{41}^t\\X_{31}^t\\X_{21}^t\end{smallmatrix}\right)} & \cokerBentmann\left({\begin{smallmatrix}B'_4&X_{43}^t&X_{42}^t\\0&B'_3&0\\0&0&B'_2\end{smallmatrix}}\right)\;,\\
  \ker\left({\begin{smallmatrix}B'_3&X_{31}^t\\0&B'_1\end{smallmatrix}}\right)\ar[ru]^<<<<r\ar[r]|\circ_*+{-{\left(\begin{smallmatrix}X_{43}^t&X_{41}^t\\0&X_{21}^t\end{smallmatrix}\right)}}    & \cokerBentmann\left({\begin{smallmatrix}B'_4&X_{42}^t\\0&B'_2\end{smallmatrix}}\right)\ar[ru]^i                    &
}
\end{equation*}
where column-wise direct sums are taken. Here $B'_1=B_1^t-\left({\begin{smallmatrix}1&0\\0&1\end{smallmatrix}}\right)=\left({\begin{smallmatrix}1&1\\1&1\end{smallmatrix}}\right)$ and $B'_j=B_j^t-{\begin{pmatrix}1\end{pmatrix}}={\begin{pmatrix}2\end{pmatrix}}$ for $2\leq j\leq 4$. This complex can be identified with
\begin{equation*}
\ZBentmann\oplus\ZBentmann/2\oplus\ZBentmann\xrightarrow{\left(\begin{smallmatrix}0&1&0\\0&0&0\\-2&0&2\\0&1&0\\0&0&0\end{smallmatrix}\right)} (\ZBentmann/2)^2\oplus\ZBentmann\oplus (\ZBentmann/2)^2\xrightarrow{\left(\begin{smallmatrix}1&0&0&1&0\\0&1&1&0&0\\0&0&1&0&1\end{smallmatrix}\right)}(\ZBentmann/2)^3\;,
\end{equation*}
the homology of which is isomorphic to $\ZBentmann/2$; a generator is given by the class of $(0,1,1,0,1)\in (\ZBentmann/2)^2\oplus\ZBentmann\oplus (\ZBentmann/2)^2$. This concludes the proof of Proposition~\ref{pro:CKBentmann}.

\end{document}